\documentclass[12pt]{article}
\usepackage{latexsym}
\usepackage{amssymb}
\usepackage{amsfonts}                     
\usepackage{amsmath}
\usepackage{amsthm}
\usepackage[all]{xy}
\usepackage{color}
\usepackage{mathrsfs}
\usepackage[utf8]{inputenc}
\usepackage[T1]{fontenc}
\usepackage{amscd}
\usepackage{amsxtra}
\usepackage{tikz-cd} 
\usepackage{tikz}
\usepackage{url}
\usepackage[a4paper, margin=2.5cm]{geometry}

\usepackage{hyperref} 
\hypersetup{
    colorlinks=true, 
    linkcolor=blue, 
    citecolor=blue, 
    urlcolor=blue    
}

\usetikzlibrary{positioning}
\newtheorem{thm}{\bf Theorem}[section]

\newtheorem{prop}[thm]{\bf Proposition}
\newtheorem{defn}[thm]{\bf Definition}
\newtheorem{defns}[thm]{\bf Definitions}

\newtheorem{rems}[thm]{\bf Remarks}
\newtheorem{exmp}[thm]{\bf Example}

\newtheorem{pb}[thm]{\bf Problem}
\newcommand{\field}[1]{\mathbb{#1}}

\newcommand{\N }{\field{N}}

\def\Sidem{\mathrm{S\text{-}idem}}
\def\Snil{\mathrm{S\text{-}nil}}
\def\Su{\mathrm{S\text{-}u}}
\def\Svnr{\mathrm{S\text{-}vnr}}
\def\Spreg{\mathrm{S\text{-}\pi\text{-}reg}}
\def\preg{\mathrm{\pi\text{-}reg}}

\def\Zero{\mathrm{Z}}
\def\u{\mathrm{u}}
\def\idem{{\rm idem}}
\def\nil{{\rm nil}}
\def\vnr{{\rm vnr}}
\def\reg{{\rm reg}}
\def\FP{{\rm FP}}
\def\z{{zero}}
\def\Max{{\rm Max}}
\def\rad{{\rm rad}}

\begin{document}
\title{On the $S$-version of some special elements in commutative rings}
\author{D. Bennis, A. Bouziri, S. D. Kumar, and T. Singh}
\date{}
\maketitle
 \begin{abstract}  In this paper, we introduce and study the $S$-versions of several fundamental elements in commutative rings. Specifically, for a commutative ring $R$  with identity and a multiplicative subset $S$, we define and investigate the notions of $S$-invertible, $S$-idempotent, $S$-von Neumann regular, and $S$-$\pi$-regular elements. We establish their basic properties, interrelations, and structural inclusions, and use them to characterize classes of rings. Special attention is given to the uniform $S$-counterparts of Boolean and $\pi$-regular rings, where we provide examples distinguishing these from their classical analogues. Several transfer results under homomorphisms and direct product constructions are established, and connections with existing $S$-counterparts (uniformly $S$-von Neumann regular, uniformly $S$-Artinian, etc.) are highlighted. Throughout the paper, we point out several open problems, offering directions for further research. \end{abstract}  
 

 {\scriptsize \textbf{Mathematics Subject Classification (2020)}:  13A15, 13A99, 13F99.}

 {\scriptsize \textbf{Key Words}:  $S$-idempotent element, $S$-von Neumann regular element, $S$-$\pi$-regular element, $S$-Boolean ring, $S$-$\pi$-regular ring.}

\section{Introduction}  

\hskip 0.5cm Throughout this paper, $R$ is a commutative ring with identity, $S$ is a multiplicative subset of $R$, and all modules are unitary. We denote by $\idem(R)$, $\vnr(R)$, $\preg(R)$, $\reg(R)$, $Z(R)$, and $\nil(R)$ the sets of idempotent elements, von Neumann regular elements, $\pi$-regular elements, regular elements, zero-divisors, and nilpotent elements of $R$, respectively.  

Recall that an element $a \in R$ is idempotent if $a^2 = a$, von Neumann regular if there exists $x \in R$ such that $a = xa^2$, $\pi$-regular if there exists $n \in \mathbb{N}^*$ such that $a^n = xa^{2n}$ for some $x \in R$, and nilpotent if $a^n = 0$ for some nonzero integer $n \in \mathbb{N}^*$. In \cite{And1}, Anderson and Badawi studied these elements and the relationships between their corresponding sets, particularly in certain special classes of rings, such as Boolean, von Neumann regular, and $\pi$-regular rings, and demonstrated that these rings can be characterized based on these relationships.   A ring $R$ is said to be Boolean (resp., von Neumann regular, $\pi$-regular) if every element of $R$ is idempotent (resp., von Neumann regular, $\pi$-regular).

   The generalization of classical notions using multiplicative subsets has attracted significant attention from researchers in homological algebra and ring theory in recent years. This approach was initiated by Anderson and Dumitrescu in 2002 when they introduced the concept of $S$-Noetherian rings, defined as rings in which every ideal $I$ is $S$-finite. That is, for some finitely generated subideal $I'$ of $I$, there exists an element $s \in S$ such that $s(I/I') = 0$ \cite{And6}. Since then, many researchers have investigated $S$-versions of various classical structures, including $S$-strong Mori domains \cite{Kim1}, $S$-coherent rings and modules \cite{Ben1}, $S$-Artinian rings and modules \cite{Sev1}, $S$-perfect rings \cite{Bou1}, $S$-flat modules \cite{Qi1}, and $S$-(FP-)injective modules \cite{Bou3, Bou4}. 

 Although many classical results have $S$-counterparts, in some cases, establishing such analogues has proven challenging. In particular, studying $S$-Noetherian rings from a module-theoretic perspective has posed significant difficulties.  This led to the development of what is known as the uniformly $S$-torsion theory, where the choice of an element from the multiplicative subset $S$ is made in a ``uniform'' manner. For example, a uniformly $S$-Noetherian ring is a ring for which there exists an element $s \in S$ such that every ideal $I$ of $R$ is $S$-finite with respect to $s$, meaning that $s(I/I') = 0$ for some finitely generated subideal $I'$ of $I$. This perspective allowed researchers to provide module-theoretic characterizations of uniformly $S$-Noetherian rings \cite{Min1}, as well as the "uniform" $S$-counterpart of many classical results (see, for instance, \cite{Zha4, Zha5, Zha6, Zha7, Zha8}).  However, even for uniformly $S$-versions, some of the uniform $S$-counterparts of classical results are not fully satisfactory. For instance, in \cite{Min1}, uniformly $S$-Noetherian rings were characterized by the assertion that the direct sum of injective modules is uniformly $S$-injective. However, a more suitable assertion seems to be that the direct sum of uniformly $S$-injective modules is uniformly $S$-injective. This remains an open question; see \cite{Bou6} for other related open problems. 
     
    It is also worth noting that recent attempts to study $S$-Noetherian and $S$-coherent rings from a module-theoretic perspective have been made by various researchers (see, for instance, \cite{Bou3, Bou4, Qi1, Zha1, Zha9}).

 In this paper, we aim to introduce and study the $S$-version of certain special elements in commutative rings. Specifically, we investigate the $S$-analogue of invertible, idempotent, von Neumann regular, and $\pi$-regular elements, and use these notions to characterize some classes of rings. Additionally, we introduce and analyze a uniform $S$-version of Boolean rings and $\pi$-regular rings, extending several related results. Furthermore, we provide examples to distinguish these new structures from the classical ones. 
 

 The paper is organized into two sections. In the first section, we introduce and study \( S \)-analogues of several classical types of elements in a ring, including invertible, von Neumann regular, and \(\pi\)-regular elements. We define the notions of \( S \)-invertibility, \( S \)-von Neumann regularity, and \( S \)-\(\pi\)-regularity, and examine their basic properties and interrelationships.

In the second section, we investigate rings in which these \( S \)-notions hold globally. Specifically, we introduce and study the classes of \( S \)-Boolean, \( S \)-von Neumann regular, and \( S \)-\(\pi\)-regular rings, along with their uniformly \( S \)-counterparts. We also explore the relationships between these different classes of rings and provide examples to distinguish them.

\section{\texorpdfstring{$S$}{S}-invertible elements and their generalizations}\label{section-2}
Recall from \cite{Ers1} that an element \( a \in R \) is called \emph{\( S \)-idempotent} if there exists \( s \in S \) such that \( sa = a^2 \). In this section, we introduce and investigate the \( S \)-versions of invertible, von Neumann regular, and \(\pi\)-regular elements.

\begin{defn}
Let \( S \) be a multiplicative subset of a commutative ring \( R \). Then,
\begin{enumerate}
\item An element \( a \in R \) is said to be \emph{\( S \)-invertible} if \( ab \in S \) for some \( b \in R \); that is, there exist \( b \in R \) and \( s \in S \) such that \( ab = s \).

\item An element \( a \in R \) is said to be \emph{\( S \)-von Neumann regular} if there exist \( b \in R \) and \( s \in S \) such that \( sa = a^2 b \).

\item An element \( a \in R \) is said to be \emph{\( S \)-\(\pi\)-regular} if there exist \( b \in R \), \( s \in S \), and \( n \in \mathbb{N} \) with \( n \geq 1 \), such that \( sa^n = a^{2n}b \).
\end{enumerate}

\end{defn}

We denote by \( \Su(R) \) (respectively, \( \Sidem(R) \), \( \Svnr(R) \), \( \Spreg(R) \)) the set of all \( S \)-invertible (respectively, \( S \)-idempotent, \( S \)-von Neumann regular, \( S \)-\(\pi\)-regular) elements of \( R \).

For each \( s \in S \), we denote by \( s\text{-}\u(R) \) (respectively, \( s\text{-}\idem(R) \), \( s\text{-}\vnr(R) \), \( s\text{-}\preg(R) \)) the set of all elements of \( R \) that are \( S \)-invertible (respectively, \( S \)-idempotent, \( S \)-von Neumann regular, \( S \)-\(\pi\)-regular) with respect to \( s \).

\begin{rems}\label{rems-follows-defn} The following assertions are straightforward to verify.
\begin{enumerate}
\item Any element \( s \in S \) is $S$-invertible, since \( s = s \cdot 1 \). 


\item Any invertible (resp., idempotent, von Neumann regular, $\pi$-regular) element of \( R \) is $S$-invertible (resp., $S$-idempotent, $S$-von Neumann regular, $S$-$\pi$-regular). The converses, however, do not hold in general. Moreover, the $\u(R)$-invertible (resp., $\u(R)$-von Neumann regular, $\u(R)$-$\pi$-regular) elements are precisely the invertible (resp., von Neumann regular, $\pi$-regular) elements of \( R \).

\item A $\u(R)$-idempotent element need not be idempotent. For example, let \( R \) be a von Neumann regular ring that is not Boolean. Then there exists \( a \in R \) which is not idempotent. By \cite[Theorem 2.2]{And1}, there exists an invertible element \( u \in R \) such that \( ua = a^2 \). Hence, \( a \) is a $\u(R)$-idempotent element that is not idempotent.

In fact, it follows from \cite[Theorem 2.2]{And1} that the $\u(R)$-idempotent elements are exactly the von Neumann regular elements of \( R \).

\item Every $S$-invertible element of \( R \) is $S$-von Neumann regular. Indeed, if \( ab = s \) for some \( a, b \in R \) and \( s \in S \), then multiplying both sides by \( a \) yields \( sa = a^2 b \).

\item Every $S$-idempotent element of \( R \) is $S$-von Neumann regular, and every $S$-von Neumann regular element is $S$-$\pi$-regular.
\end{enumerate} Thus, we have the following inclusions: \[S \subseteq \Su(R) , \quad 
 \idem(R) \cup S \subseteq \Sidem(R) \subseteq S \cup \Zero(R),
\] \[
\vnr(R)\cup  \Su(R) \cup \Sidem(R) \subseteq \Svnr(R) \subseteq \Su(R) \cup \Zero(R),
\] \[ \preg(R) \cup \Svnr(R) \cup \Snil(R) \subseteq \Spreg(R) \subseteq \Su(R) \cup \Zero(R). \] where 
\[
\Snil(R) = \{ r \in R \mid sr^n = 0 \text{ for some } n \in \mathbb{N}^* \text{ and } s \in S \}.
\]

\end{rems}

The following Example~\ref{2exmp-} shows that an \( S \)-invertible element does not necessarily belong to \( S \).

\begin{exmp}\label{2exmp-}  
  Let $R$ be an integral domain, and let $p$ be a prime element; let $s=p^2$ and set $S=\{s^n/n \in  \N\}$. Then $p$ is an $S$-invertible element which is not in $S$. 
 \end{exmp}
Recall from \cite{Ers1} that an element \( a \in R \) is said to be \( S \)-zero if there exists \( s \in S \) such that \( sa = 0 \).

Let \( P \) be a prime ideal of \( R \). An element \( r \in R \) is said to be \( P \)-invertible, \( P \)-idempotent, and so forth, if it is \( (R \setminus P) \)-invertible, \( (R \setminus P) \)-idempotent, and so on.

For a ring $R$, we denote by $\Max(R)$ the set of all maximal ideals of $R$. The following proposition highlights the connection between the classical notions and their \( m \)-relative counterparts, where \( M \) ranges over \( \Max(R) \).

\begin{prop}
The following assertions hold:
\begin{enumerate}
    \item \( \{0\} = \bigcap_{M \in \Max(R)} M\text{-}\z(R) \);
    \item \( \u(R) = \bigcap_{M \in \Max(R)} M\text{-}\u(R) \);
    \item \( \nil(R) = \bigcap_{M \in \Max(R)} M\text{-}\nil(R) \);
    \item \( \idem(R) \subseteq \bigcap_{M \in \Max(R)} M\text{-}\idem(R) \subseteq \vnr(R) \);
    \item \( \vnr(R) = \bigcap_{M \in \Max(R)} M\text{-}\vnr(R) \).
\end{enumerate}
\end{prop}

\begin{proof}
(1)  Let \( a \in \bigcap_{M \in \Max(R)} M\text{-}\z(R) \). For each maximal ideal \( M \), there exists an element \( t_M \notin M \) such that \( t_M a = 0 \). Since the ideal generated by the family \( \{ t_M : M \in \Max(R) \} \) is equal to \( R \), there exist elements \( r_1, \dots, r_n \in R \) and maximal ideals \( M_1, \dots, M_n \in \Max(R) \) such that
\[
1 = r_1 t_{M_1} + \cdots + r_n t_{M_n}.
\]
It follows that
\[
a = r_1 t_{M_1} a + \cdots + r_n t_{M_n} a = 0.
\]

(2) Let \( a \in \bigcap_{M \in \Max(R)} M\text{-}\u(R) \). For each maximal ideal \( M \), there exist \( t_M \notin M \) and \( x_M \in R \) such that \( a x_M = t_M \). As in (1), the family \(\{ t_M : M \in \Max(R) \} \) generate the unit ideal, so
\[
1 = r_1 t_{M_1} + \cdots + r_n t_{M_n}
\]
for some \( r_i \in R \) and \( M_i \in \Max(R) \). Hence,
\[
1 = r_1 a x_{M_1} + \cdots + r_n a x_{M_n} = a(r_1 x_{M_1} + \cdots + r_n x_{M_n}),
\]
which shows that \( a \in \u(R) \).

(3) Let \( a \in \bigcap_{M \in \Max(R)} M\text{-}\nil(R) \). For each maximal ideal \( M \), there exists \( t_M \notin M \) and \( p_M \in \mathbb{N} \) such that \( t_M a^{p_{M}} = 0 \). Since  the family \(\{ t_M : M \in \Max(R) \} \) generate the unit ideal, there exist \( r_1, \dots, r_n \in R \) and \( M_1, \dots, M_n \in \Max(R) \) such that
\[
1 = r_1 t_{M_1} + \cdots + r_n t_{M_n}.
\]
Set \( N = p_{M_1} + \cdots + p_{M_n} \). Then,
\[
a^N = a^N (r_1 t_{M_1} + \cdots + r_n t_{M_n}) = \sum r_i a^N t_{M_i} = 0,
\]
since each \( t_{M_i} a^{p_{M_i}} = 0 \), and \( a^N = a^{q} a^{p_{M_i}} \) for some \( q \ge 0 \). Hence \( a \in \nil(R) \).

(4) The first inclusion is clear. Let \( a \in \bigcap_{M \in \Max(R)} M\text{-}\idem(R) \). Then for each \( M \), there exists \( t_m \notin M \) such that \( t_M a = a^2 \). As above, we write
\[
1 = r_1 t_{M_1} + \cdots + r_n t_{M_n}.
\]
Then,
\[
a = r_1 t_{M_1} a + \cdots + r_n t_{M_n} a = r_1 a^2 + \cdots + r_n a^2 = (r_1 + \cdots + r_n) a^2,
\]
so \( a \in \vnr(R) \). If \( a \) is regular, then
\[
1 = (r_1 + \cdots + r_n) a.
\]

(5) This follows by using a similar argument as in the previous statements.

\end{proof}

Recall from \cite{Yil1} that an ideal \( P \) of a ring \( R \) is said to be \( S \)-prime if \( P \cap S = \emptyset \) and there exists a fixed \( s \in S \) such that, whenever \( ab \in P \) for some \( a, b \in R \), we have either \( sa \in P \) or \( sb \in P \). 

An ideal \( M \) of \( R \) is called \( S \)-maximal if \( M \cap S = \emptyset \) and there exists a fixed \( s \in S \) such that, for any ideal \( I \) of \( R \) with \( M \subseteq I \), either \( sI \subseteq M \) or \( I \cap S \neq \emptyset \).

Recall also from \cite{Sev1} that a ring \( R \) is called an \( S \)-integral domain if there exists a fixed \( s \in S \) such that whenever \( ab = 0 \) for some \( a, b \in R \), then either \( sa = 0 \) or \( sb = 0 \). Note that every integral domain is an \( S \)-integral domain. Moreover, \( R \) is an \( S \)-integral domain if and only if the zero ideal \( (0) \) is an \( S \)-prime ideal.

Furthermore, recall from \cite[Definition 9]{Yil1} that \( R \) is called an \( S \)-field if \( (0) \) is an \( S \)-maximal ideal. By definition, every field is also an \( S \)-field, and every \( S \)-field is an \( S \)-integral domain.

We now present the following result, which extends the classical characterization of a field: a ring $R$ is a field if and only if every nonzero element of $R$ is invertible.

\begin{prop}
The following assertions are equivalent:
\begin{enumerate}
\item \( R \) is an \( S \)-field.
\item There exists \( s \in S \) such that every element of \( R \) is either \( S \)-invertible or \( S \)-zero with respect to \( s \).
\end{enumerate}
\end{prop}

\begin{proof}
$1. \Rightarrow 2.$ Assume that \( (0) \) is \( S \)-maximal with respect to some \( s \in S \). Let \( a \in R \) be a non-\( S \)-zero element. Since \( (0) \subseteq Ra \), and \( (0) \) is \( S \)-maximal, we have either \( Ra \cap S \neq \emptyset \) or \( sRa = 0 \). But \( a \) is not \( S \)-zero, so \( sa \neq 0 \), and thus \( sRa \neq 0 \). Hence, \( Ra \cap S \neq \emptyset \), i.e., there exists \( r \in R \) and \( t \in S \) such that \( ra = t \). This shows that \( a \) is \( S \)-invertible.

$2. \Rightarrow 1.$ Let \( I \) be an ideal of \( R \) such that \( I \cap S = \emptyset \). Then no element of \( I \) is \( S \)-invertible. By assumption (2), each element of \( I \) must be \( S \)-zero with respect to some fixed \( s \in S \). That is, for all \( a \in I \), we have \( sa = 0 \), so \( sI = 0 \). Therefore, \( (0) \) is \( S \)-maximal with respect to \( s \), completing the proof.
\end{proof}

We define a subset \( A \subseteq R \) to be \( S \)-torsion if every element \( a \in A \) is \( S \)-zero.

\begin{prop}\label{prop-on-interse-of-some-special-subset} Let \( R \) be a commutative ring, and let $S$ be a multiplicative subset of $R$. 
\begin{enumerate}
\item If \( a, b \in R \) with \( ab \in \Su(R) \), then \( a \in \Su(R) \).

\item If \( e \) is \( S \)-idempotent with respect to \( s \in S \), then so is \( s - e \).
\item Let \( a \in R \). If \( sa = a^2b \) (resp., \( sa^n = a^{2n}b \)) for some \( b \in R \), \( s \in S \), and \( n \geq 1 \), then \( ab \in \Sidem(R) \) (resp., \( a^n b \in \Sidem(R) \)).
   
    \item The subsets \( \Su(R) \), \( \Sidem(R) \), \( S\text{-}\mathrm{vnr}(R) \), and \( S\text{-}\pi\text{-}\mathrm{reg}(R) \) are multiplicatively closed.
    
    \item The intersection \( \Svnr(R) \cap \Snil(R) \) is \( S \)-torsion. Consequently, so is the subset \( \Sidem(R) \cap \Snil(R) \).
     
    \item Let \( s, t \in S \). If the set \( \{s^n : n \in \mathbb{N} \} \) is finite, then the set \( s\text{-}\vnr(R) \cap t\text{-}\nil(R) \) is uniformly \( S \)-torsion, and so is \( s\text{-}\idem(R) \cap t\text{-}\nil(R) \).
\end{enumerate}
\end{prop}

\begin{proof}

1. Let \( s \in S \) and \( c \in R \) such that \( (ab)c = t \), where \( t \in S \). Thus, \( a(bc) = t \in S \), which implies that \( a \) is $S$-invertible in \( R \), and hence \( a \in \Su(R) \).

2. If \( se = e^2 \), then:
\[
(s - e)^2 = s^2 - 2se + e^2 = s(s - e).
\]
Thus, \( s - e \) is also \( S \)-idempotent.

3. Let \( a \in R \) be an \( S \)-von Neumann regular element with respect to some \( s \in S \), i.e., there exists \( b \in R \) such that \( sa = a^2b \). Then
\[
s(ab) = (sa)b = a^2b^2 = (ab)^2,
\]
thus \( ab \in \Sidem(R) \). The "respectively" part follows similarly.

\medskip

4. Let \( a, b \in R \) be two elements that are von Neumann regular with respect to \( s \) and \( t \), respectively. Then there exist \( x, y \in R \) such that \( sa = a^2x \) and \( tb = b^2y \). Then
\[
st(ab) = (sa)(tb) = a^2x b^2y = (ab)^2(xy),
\]
which shows that \( ab \in \Svnr(R) \). 

A similar argument applies to the other subsets.
\medskip

5. Let \( a \in \Svnr(R) \cap \Snil(R) \). Then there exist \( s, t \in S \), \( x \in R \), and \( n \in \mathbb{N} \) such that \( sa = a^2x \) and \( ta^n = 0 \). Then
\[
 ts^{n-1}a = txs^{n-2}a^2 = tx^2s^{n-3}a^3 = \cdots = tx^{n-1}a^n = 0.
\]
Hence \( a \) is an \( S \)-zero, and the intersection is \( S \)-torsion. 

\medskip

6. This follows similarly to (3), using the finiteness of the set \( \{s^n : n \in \mathbb{N} \} \).
\end{proof}

 As an \( S \)-analogue of \cite[Theorem 2.2]{And1}, which provides a characterization of von Neumann regular elements, we state the following theorem.
 
\begin{thm}\label{thm-s-von-ele-char}
Let \( R \) be a commutative ring. The following statements are equivalent for \( a \in R \):
\begin{enumerate}
    \item \( a \in \Svnr(R) \).
    \item There exist \( s \in S \) and \( u \in \Su(R) \) such that \( sa = a^2u \).
    \item There exist \( s \in S \), \( u \in \Su(R) \), and \( e \in \Sidem(R) \) such that \( sa = ue \).
    \item There exist \(s \in S \) and \( b \in \Svnr(R) \setminus \{a\} \) such that \( sab = 0 \) and \( sa + b \in \Su(R) \).
    \item There exist \( s \in S \) and \( b \in R \) such that \( sab = 0 \) and \( sa + b \in U(R) \).
\end{enumerate}
\end{thm}

\begin{proof}
\(1.  \Rightarrow  2.\) Let \( s \in S \) and \( x \in R \) such that \( sa = xa^2 \). Let \( e = ax \). By proposition \ \ref{prop-on-interse-of-some-special-subset}, we have \( se = e^2 \) and \( s(s-e) = (s-e)^2 \). Moreover, we have \( a(s - e) = 0 \). Let \( u = ex + s - e \) and \( v = a + s - e \). Then, we have \( uv = s^2 \), so \( u \in \Su(R) \), and
\[
a^2u = a^2(ex + s - e) = a^2ex + a^2(s - e) = a^2ex = sae = sa^2x = s^2a.
\]

\(2. \Rightarrow  3.\) 
Assume that \( sa = a^2u \) for some \( s \in S \) and \( u \in \Su(R) \). Let \( t \in S \) and \( v \in R \) such that \( uv = t \). Set \( e = au \), so we have \( se = e^2 \in \idem(R) \), and
\[
ve = v(au) = t a.
\]

 \(3. \Rightarrow  4.\)
Assume (3) holds, i.e., \( sa = ue \) for some \( u \in \Su(R) \) and \( e \in \Sidem(R) \). Let \( v \in R \) and \( t \in S \) such that \( uv = t \), and let \( r \in S \) such that \( re = e^2 \). Then, for \( b = u(r - e) \), we have
\[
b^2v = u^2(r - e)^2v = tru(r - e) = trb,
\]
since \( (r - e)^2 = r(r - e) \) by Proposition \ref{prop-on-interse-of-some-special-subset}(2). Hence, \( b \in \Svnr(R) \). Moreover,
\[
sab = (ue)(u(r - e)) = u^2e(r - e) = 0,
\]
and
\[
sa + b = ue + ur - ue = ru.
\]
Thus, \( sa + b \in \Su(R) \), as desired.

 \(4. \Rightarrow 5.\) 
This is clear.

\(5. \Rightarrow 1.\) Assume \( sa + b = u \in U(R) \) and \( sab = 0 \) for some \( b \in R \), \( s \in S \). Let \( v \in R \) and \( t \in S \) such that \( uv = t \). Then,
\[
sau = sa(sa + b) = s^2a^2 + sab = s^2a^2.
\]
Thus, 
\[
a^2(s^2v) = (s^2a^2)v = (sau)v = sta,
\]
so \( a \in \Svnr(R) \), as desired.
\end{proof} 
  
Recall from \cite[Lemma~4]{Rap1} that if \( a \) is a von Neumann regular element, then there exists a unique element \( x \in R \) such that \( a = a^{2}x \) and \( x = x^{2}a \). Theorem~\ref{thm-S-von-regualr-unique-decompostion} can be viewed as an \( S \)-analogue of this result.

   \begin{thm}\label{thm-S-von-regualr-unique-decompostion}
Let \( R \) be a commutative ring and let \( a \in R \) be an \( S \)-von Neumann regular element. Then, there exist \( s \in S \) and \( x \in R \) such that \( a^2x = sa \) and \( x^2a = sx \). Moreover, if \( y \in R \) is another such element, then \( s^4y = s^4x \). In particular, if \( s \) is a regular element, then \( x = y \), and hence \( x \) is unique.
\end{thm}

\begin{proof}
Let \( a \in R \) be an \( S \)-von Neumann regular element. Then there exist \( b \in R \) and \( s \in S \) such that \( sa = ba^2 \). Set \( x = ab^2 \). We compute:
\[
xa^2 = ab^2a^2 = a(ba)^2 = (ba)^2a = s^2a, \quad \text{so } s^2a = xa^2,
\]
and
\[
ax^2 = a(ab^2)^2 = a^2b^4 = s^2b^2 = s^2x, \quad \text{so } s^2x = ax^2.
\]
Thus, such an \( x \) exists.

Now, suppose \( y \in R \) also satisfies \( s^2a = ya^2 \) and \( s^2y = ay^2 \). Then:
\[
s^4y = s^2a y^2 = x a^2 y^2 = x (a^2 y) y = s^2 x y a =  a x^2 y a = (a^2 y) x^2 = s^2 a x^2 = s^4 x.
\]
Hence, \( s^4 y = s^4 x \). In particular, if \( s \) is a regular element, then cancellation implies \( y = x \), and the uniqueness of \( x \) follows.
\end{proof}


It was shown in \cite[Theorem 2.1]{And1} (respectively, \cite[Theorem 4.1]{And1}) that \( \vnr(R) \) (respectively, \( \preg(R) \)) is always closed under multiplication. Moreover, the additive closure of \( \vnr(R) \) was investigated in \cite[Theorems 2.9 and 2.11]{And1}. The key results are: if \( \vnr(R) \) is closed under addition, then \( R \) is reduced \cite[Theorem 2.9]{And1}; and if \( 2 \in \u(R) \), then \( \vnr(R) \) is closed under addition if and only if the sum of any four units in \( R \) is a von Neumann regular element \cite[Theorem 2.11]{And1}.

These observations naturally lead to the following question:

\begin{pb}\label{problem1}
When is \( \Svnr(R) \) closed under addition? In other words, when is it a subring of \( R \)?
\end{pb}

We first recall that a ring \(R\) is called  \(S\)-reduced if it has no nonzero \(S\)-nilpotent elements; that is, \(\Snil(R) = 0\). An element \(a \in R\) is said to be \(S\)-nilpotent if there exist \(s \in S\) and \(n \in \mathbb{N}^*\) such that \(sa^n = 0\) \cite{Ers1}.

Moreover, a ring \( R \) is said to be uniformly \( S \)-reduced if \( \nil(R) \) is uniformly \( S \)-torsion \cite{Kim2}.

Notice that every reduced ring is both uniformly \( S \)-reduced and \( S \)-reduced, and that if \( R \) is \( S \)-reduced, then \( S \subseteq \reg(R) \). In our setting, we are interested in the following related notion.

\begin{defn}
A ring \( R \) is said to be weakly \( S \)-reduced if \( \nil(R) = 0 \) is \( S \)-torsion.
\end{defn}

Obviously, any reduced ring is weakly \( S \)-reduced; however, the converse does not hold in general. Indeed, it suffices to consider a uniformly \( S \)-reduced ring that is not reduced. In Example~\ref{exmp-weakly-s-reduced-not-u-s-reduced}, we provide an example of a weakly \( S \)-reduced ring that is not uniformly \( S \)-reduced. 

\begin{exmp}\label{exmp-weakly-s-reduced-not-u-s-reduced} 
Let \( R \) be a commutative ring and \( S \) a multiplicative subset of \( R \) such that:
\begin{enumerate}
    \item \( R_S \) is von Neumann regular.
    \item There exists an \( S \)-torsion \( R \)-module \( M \) that is not uniformly \( S \)-torsion.
\end{enumerate}
See \cite[Example~3.15]{Zha4} for such an example. Then the trivial extension \( R \ltimes M \) is a weakly \( S \ltimes \{0\} \)-reduced ring that is not uniformly \( S \ltimes \{0\} \)-reduced.
\end{exmp}

\begin{proof}
Let \( x \in R \) be such that \( x^n = 0 \) for some \( n \in \mathbb{N} \). Then 
\[
\frac{x^n}{1} = \left( \frac{x}{1} \right)^n = \frac{0}{1}
\]
in \( R_S \). Since \( R_S \) is reduced, it follows that \( \frac{x}{1} = \frac{0}{1} \), so there exists \( s \in S \) such that \( sx = 0 \). Thus, \( R \) is weakly \( S \)-reduced. 

Since \( M \) is \( S \)-torsion and \( \nil(R \ltimes M) = \nil(R) \ltimes M \), it follows that \( R \ltimes M \) is also weakly \( S \ltimes \{0\} \)-reduced. However, \( R \ltimes M \) is not uniformly \( S \ltimes \{0\} \)-reduced, because \( M \) is not uniformly \( S \)-torsion.
\end{proof}

\begin{prop}\label{prop-s-red-iff-s-torR-is-s-tor} The following assertions are equivalent: 
\begin{enumerate}
\item $R$ is weakly $S$-reduced.
\item $\Snil(R)$ is $S$-torsion.
\end{enumerate}
\end{prop}
\begin{proof}
$(1)\Rightarrow (2)$ Let \( a \in \Snil(R) \). Then, there exist \( s \in S \) and \( n \in \mathbb{N} \) such that \( sa^n = 0 \). Therefore, \( (sa)^n = 0 \), which implies that \( sa \in \nil(R) \). Thus, there exists \( t \in S \) such that \( ts a = 0 \). Therefore, \( \Snil(R) \) is \( S \)-torsion. 

$(2)\Rightarrow (1)$  This follows from the fact that $\nil(R)\subseteq \Snil(R)$. 
\end{proof}

It follows from Proposition~\ref{prop-s-red-iff-s-torR-is-s-tor} that every \( S \)-reduced ring is weakly \( S \)-reduced. However, Example~\ref{exmp-weakly-s-reduced-not-u-s-reduced} demonstrates that the converse does not hold in general.

Using arguments similar to those in the proof of Proposition~\ref{prop-s-red-iff-s-torR-is-s-tor}, we can establish the following result, Proposition~\ref{prop-u-s-red-iff-s-torR-is-u-s-tor}.

\begin{prop}\label{prop-u-s-red-iff-s-torR-is-u-s-tor} The following assertions are equivalent: 
\begin{enumerate}
\item $R$ is uniformly $S$-reduced. 
\item There exists \( s \in S \) such that \( s\text{-}\nil(R) \) is uniformly \( S \)-torsion. 
\end{enumerate}
\end{prop}

\begin{thm} The following assertions hold true: 
\begin{enumerate}
\item If \( \Svnr(R) \) is a subring of \( R \), then \( R \) is weakly \( S \)-reduced.

\item Assume that there exists \( s \in S \) such that \( \{s^n \mid n \in \mathbb{N} \} \) is finite and \( s\text{-}\vnr(R) \) is a subring of \( R \), then \( R \) is uniformly \( S \)-reduced.
\end{enumerate}
\end{thm}
\begin{proof}
1. Let \( a \in \nil(R) \), and let $n \in \N$ such that $a^n=0$. Hence, \( 1 + a \in U(R) \subseteq \vnr(R) \subseteq \Svnr(R) \). Since \( \Svnr(R) \) is closed under addition, we have:
\[
a = -1 + (1 + a) \in \Svnr(R).
\]
Hence, \( a \in \nil(R) \cap \Svnr(R) \subseteq \Snil(R) \cap \Svnr(R) \). By Proposition \ref{prop-on-interse-of-some-special-subset}, there exists \( t \in S \) such that \( ta = 0 \) in \( R \). Therefore, \( R \) is weakly \( S \)-reduced.

2. Let \( a \in \nil(R) \). Then \( 1 + a \in U(R) \subseteq \vnr(R) \subseteq \Svnr(R) \). Since \( \Svnr(R) \) is closed under addition, we again have:
\[
a = -1 + (1 + a) \in \Svnr(R).
\]
Thus, \( a \in \nil(R) \cap s\text{-}\vnr(R) = 1\text{-}\nil(R) \cap s\text{-}\vnr(R) \). Consequently, \( \nil(R) = 1\text{-}\nil(R) \cap s\text{-}\vnr(R) \), which is uniformly \( S \)-torsion by Proposition \ref{prop-on-interse-of-some-special-subset}. 
\end{proof}

\begin{prop}\label{prop-if-2-u-s-inv-then-S-von-is-sum-of-2-S-inve}
Let \( R \) be a commutative ring with \( 2 = 1 + 1 \in \Su(R) \), and let \( a \in R \) be an \( S \)-von Neumann regular element. Then there exists \( s \in S \) such that \( sa \) is the sum of two \( S \)-invertible elements of \( R \).
\end{prop}

\begin{proof}
Let \( x \in R \) and \( k \in S \) be such that \( 2x = k \). Let \( a \in \Svnr(R) \). Then, by Theorem~\ref{thm-s-von-ele-char}, we have \( sa = ue \) for some \( u \in \Su(R) \) and \( e \in \Sidem(R) \). Let  \( t \in S \) be such that  \( te = e^2 \). Note that
\[
(2e - t)^2 = 4e^2 - 4te + t^2 = t^2,
\]
so \( w = 2e - t \in \Su(R) \), since \( w^2 = t^2 \in S \). Hence, we have:
\[
ke = xw + xt,
\]
and therefore
\[
ksa = kue = u(ke) = u(xw + xt) = xuw + xut.
\]
 Thus, \( ksa \) is the sum of two \( S \)-invertible elements of \( R \), and the result follows.
\end{proof}

 As a partial answer to Problem~\ref{problem1}, we establish the following result, Theorem~\ref{thm-Svnr(R)-subring}, which can be regarded as an extension of \cite[Theorem~2.11]{And1}.

\begin{thm}\label{thm-Svnr(R)-subring}
Let \( R \) be a commutative ring with \( 2 \in \Su(R) \). Then the following statements are equivalent:
\begin{enumerate}
    \item \( \Svnr(R) \) is a subring of \( R \).
    \item The sum of any four \( S \)-invertible elements is an \( S \)-von Neumann regular element of \( R \).
    \item \( u(s + k) + v(t + m) \in \Svnr(R) \) whenever \( s, t \in S \) and \( u, v, k, m \in \Su(R) \) with \( k^2, m^2 \in S \).

\end{enumerate}
\end{thm}

\begin{proof}
\(1.\Rightarrow 2.\) This follows directly from the inclusion \( \Su(R) \subseteq \Svnr(R) \), as noted in Remarks~\ref{rems-follows-defn}(5).

\(2.\Rightarrow 3.\) This follows trivially by Proposition~\ref{prop-on-interse-of-some-special-subset}(4), since the given element is a sum of four \( S \)-invertible elements.

\((3) \Rightarrow (1)\) By Proposition~\ref{prop-on-interse-of-some-special-subset}(4), it suffices to show that \( \Svnr(R) \) is closed under addition. Let \( a, b \in \Svnr(R) \). We aim to prove that \( a + b \in \Svnr(R) \). Let \( x \in R \) and \( k \in S \) such that \( 2x = k \).

By the proof of Proposition~\ref{prop-if-2-u-s-inv-then-S-von-is-sum-of-2-S-inve}, there exist \( s_a, s_b, t_a, t_b \in S \), and \( u_a, u_b, w_a, w_b \in \Su(R) \) such that $ w_a^2= t_a^2,  w_b^2=t_b^2,$ and   
\[
ks_a a = x u_a w_a + x u_a t_a, \quad ks_b b = x u_b w_b + x u_b t_b.
\]

Therefore,
\[
ks_a s_b (a + b) = s_b (ks_a a) + s_a (ks_b b) = x s_b u_a (w_a + t_a) + x s_a u_b (w_b + t_b).
\]
By (3), \(ks_a s_b (a + b) \in  \Svnr(R) \). Since \( ks_a s_b \in S \), we conclude that \( a + b \in \Svnr(R) \). 
\end{proof}

The rest of this section is devoted to examining the transfer of the different notions defined above along various ring constructions.

Let \( f : R \to R' \) be a ring homomorphism, and let \( S \) be a multiplicative subset of \( R \). It is straightforward to observe that \( f(S) \) forms a multiplicative subset of \( R' \) whenever \( 0 \notin f(S) \).

\begin{prop}\label{prop-homorphism-change-result} Let $f : R \to R'$ be a homomorphism of commutative rings and let $S$ be a multiplicative subset of $R$. Then, 
\begin{enumerate}
\item $f(\Su(R)) \subseteq f(S)\text{-}\mathrm{u}(R')$; 
\item $f(\Sidem(R)) \subseteq f(S)\text{-}\mathrm{idem}(R')$;
\item $f(\Svnr(R)) \subseteq f(S)\text{-}\mathrm{vnr}(R')$; 
\item $f(\Spreg(R)) \subseteq f(S)\text{-}\pi\text{-}\mathrm{reg}(R')$.
\end{enumerate} Moreover, if $f$ is surjective, then all the above inclusions are equality.
\end{prop} 
\begin{proof}
Straightforward.
\end{proof}

\noindent
Next, we state a uniformly \( S \)-version of Proposition~\ref{prop-homorphism-change-result}.

\begin{prop}\label{prop-homorphism-change-uni-result} Let $f : R \to R'$ be a homomorphism of commutative rings and let $S$ be a multiplicative subset of $R$. Let $s\in S$.  Then, 
\begin{enumerate}
\item $f(s\text{-}\u(R)) \subseteq f(s)\text{-}\mathrm{u}(R')$; 
\item $f(s\text{-}\idem(R)) \subseteq f(s)\text{-}\mathrm{idem}(R')$;
\item $f(s\text{-}\vnr(R)) \subseteq f(s)\text{-}\mathrm{vnr}(R')$; 
\item $f(s\text{-}\preg(R)) \subseteq f(s)\text{-}\pi\text{-}\mathrm{reg}(R')$.
\end{enumerate} Moreover, if $f$ is surjective, then all the above inclusions are equality.
\end{prop} 

Let $R = \prod_{i\in I} R_i$ be the direct product of a family of commutative rings $(R_i)_{i\in I}$. Let $e_i$ denotes the element of $R$   whose the $i$-th component is $1$ and the others are $0$. Notice that if $S$ is a multiplicative subset of $R$ then, for each $i\in I$, the subset $e_iS:=\{r_i\in R_i / \pi_i(s)=r_i \text{ for some } s\in S\}$, where $\pi_i: R\to R_i$ denote the canonical projection, is  multiplicative subset of $R_i$ (which of course maybe contains zero). Conversely,  if for each $i$, $S_i$ is a multiplicative subset of $R_i$, then $S = \prod_{i\in I} S_i$  is a multiplicative subset of $R$.  

\begin{prop}\label{prop-product-change-result} 
Let $R$ be the direct product of a family of commutative rings $\{R_i\}_{i\in I}$ and let $S$ be a multiplicative subset of $R$.  Then we have : 
\begin{enumerate}
\item $\Su\left(\prod R_i\right) = \prod e_i\Su(R_i)$. 
\item $\Sidem\left(\prod R_i\right) = \prod e_i\Sidem(R_i)$. 
\item $\Svnr\left(\prod R_i\right) = \prod e_iS\text{-}\vnr(R_i)$. 
\item If $I$ is finite, then $\Spreg\left(\prod R_i\right) = \prod e_i\Spreg(R_i)$,
\item If $I$ is finite, then $\Snil\left(\prod R_i\right) = \prod e_i\Snil(R_i)$. In particular, $R$ is (weakly) $S$-reduced if and only if each $R_i$ is (weakly) $S_i$-reduced.
\end{enumerate}
\end{prop}

\begin{proof}
The proof is routine. 
\end{proof}   
Now, we state a uniformly \( S \)-version of Proposition~\ref{prop-product-change-result}.
\begin{prop}\label{prop-product-change-uni-result}
Let $R$ be the direct product of a family of commutative rings $\{R_i\}_{i\in I}$ and let $S$ be a multiplicative subset of $R$. Then we have :
\begin{enumerate}
\item $s\text{-}\u(\prod R_i)= \prod s_i\text{-}\u(R_i)$;        
\item $s\text{-}\idem(\prod R_i)= \prod s_i\text{-}\idem(R_i)$; 
\item $s\text{-}\vnr(\prod R_i)= \prod s_i\text{-}\vnr(R_i)$; 
\item If $I$ is finite, then $s\text{-}\preg(\prod R_i)= \prod s_i\text{-}\preg(R_i)$.
\item If $I$ is finite, then $s\text{-}\nil(\prod R_i)= \prod s_i\text{-}\nil(R_i)$. In particular, $R$ is uniformly $S$-reduced if and only if each $R_i$ is uniformly $S_i$-reduced.
\end{enumerate}
\end{prop}
\section{\texorpdfstring{(Uniformly) $S$-Boolean  and uniformly $S$-$\pi$-regular rings}{(Uniformly) S-Boolean and uniformly S-pi-regular rings}} Following Zhang~\cite{Zha1}, a ring $R$ is said to be \emph{uniformly $S$-von Neumann regular} if there exists an element $s \in S$ such that every element of $R$ is $S$-von Neumann regular with respect to $s$. Moreover, by \cite[Proposition~3.10]{Zha1}, one has that $R_S$ is von Neumann regular if and only if every element of $R$ is $S$-von Neumann regular. 

Similarly, at first glance, it may seem natural to consider the class of rings in which every element is $S$-$\pi$-regular. However, this situation occurs precisely when the localization $R_S$ is $\pi$-regular.  Indeed, if every element \(a \in R\) is \(S\)-\(\pi\)-regular, then by definition there exist \(s \in S\), \(b \in R\), and \(n \geq 1\) such that 
$sa^n = a^{2n}b$. Then, for any \(t \in S\), we have 
\[
\Bigl(\tfrac{a}{t}\Bigr)^n 
= \tfrac{sa^n}{st^n} 
= \tfrac{a^{2n}b}{st^n} 
= \Bigl(\tfrac{a}{t}\Bigr)^{2n} \tfrac{t^n b}{s},
\]
which shows that \(\tfrac{a}{t}\) is \(\pi\)-regular in \(R_S\). Conversely, suppose that \(R_S\) is \(\pi\)-regular. Then for each \(a \in R\) there exist \(n \geq 1\) and \(\tfrac{b}{t} \in R_S\) such that 
\[
\Bigl(\tfrac{a}{1}\Bigr)^n = \Bigl(\tfrac{a}{1}\Bigr)^{2n}\tfrac{b}{t}.
\] 
Clearing denominators, we obtain $sa^n = a^{2n}b'$, for some \(s \in S\) and \(b' \in R\). Hence \(a\) is \(S\)-\(\pi\)-regular.

Thus, the property that "every element is \(S\)-\(\pi\)-regular" does not provide anything essentially new. In contrast, the situation is different for \(S\)-idempotent elements, since, as shown in Example~\ref{uni-bool}, the fact that every element of \(R\) is \(S\)-idempotent does not imply that \(R_S\) is Boolean.  

These observations motivate the following definitions.

\begin{defns}\label{section-3-defns}
Let \(S\) be a multiplicative subset of a commutative ring \(R\). 
\begin{enumerate}
\item The ring \(R\) is called \emph{\(S\)-Boolean} if every element of \(R\) is \(S\)-idempotent. It is said to be \emph{uniformly \(S\)-Boolean} if there exists \(s \in S\) such that every non-idempotent element of \(R\) is \(S\)-idempotent with respect to \(s\).  

\item The ring \(R\) is called \emph{uniformly \(S\)-\(\pi\)-regular} if there exists \(s \in S\) such that every  element of \(R\) is \(S\)-\(\pi\)-regular with respect to \(s\).  
\end{enumerate}
\end{defns}

\begin{rems}\label{rems-on-defn} 
The following assertions can be verified easily:
\begin{enumerate}

\item Any Boolean ring is a uniformly \( S \)-Boolean ring (with respect to \(1\)), and any uniformly \( S \)-Boolean ring is an \( S \)-Boolean ring.
 
\item Any uniformly \( S \)-Boolean ring is uniformly \( S \)-von Neumann regular, and any uniformly \( S \)-von Neumann regular ring is uniformly \( S \)-\( \pi \)-regular.
 
\item Any \( \pi \)-regular ring is uniformly \( S \)-\( \pi \)-regular.
\end{enumerate}
\end{rems}
 
Below we provide examples showing that the converses of the assertions in Remarks~\ref{rems-on-defn} do not hold in general.  
We first give an example of a uniformly \( S \)-Boolean ring which is not Boolean.

\begin{exmp}\label{uni-bool}
Consider the ring \( R = \mathbb{Z}_6 \), and let \( S = \{ \bar{1}, \bar{5}\} \) be a subset of \( R \).  
It is evident that the idempotent elements of \( R \) are \( \bar{0}, \bar{1}, \bar{3}, \bar{4} \), while \( \bar{2} \) and \( \bar{5} \) are not idempotent.  
Take \( s=\bar{5} \). Then 
\[
(\bar{2})^2 = \bar{2}s \qquad \text{and} \qquad (\bar{5})^2 = \bar{5}s.
\]
Therefore, every non-idempotent element of \( R \) is an \( S \)-idempotent with respect to \( s=\bar{5} \).  
Thus \( R \) is a uniformly \( S \)-Boolean ring.
\end{exmp}
 
Next, we give an example of a uniformly \( S \)-Boolean ring which is not an \( S \)-Boolean ring.

\begin{exmp}\label{exmp-S-vN-reg-not-u-S-vN-reg}
Consider the ring \( R = \mathbb{Z}_3 \times \mathbb{Z}_3 \) and the multiplicatively closed subset 
\[
S = \{ (\bar{1}, \bar{1}), (\bar{1}, \bar{2}), (\bar{2}, \bar{1}), (\bar{2}, \bar{2}) \}.
\]
It is easy to verify that the non-idempotent elements of \( R \) are 
\[
\alpha_1 = (\bar{2}, \bar{0}), \quad 
\alpha_2 = (\bar{0}, \bar{2}), \quad 
\alpha_3 = (\bar{1}, \bar{2}), \quad 
\alpha_4 = (\bar{2}, \bar{1}), \quad 
\alpha_5 = (\bar{2}, \bar{2}).
\]
For the elements \( \alpha_1 \) and \( \alpha_3 \), we take \( s = (\bar{1}, \bar{2}) \in S \), and observe that 
\[
\alpha_1^2 = s \alpha_1, \qquad \alpha_3^2 = s \alpha_3.
\]
Similarly, for \( \alpha_2 \) and \( \alpha_4 \), let \( s' = (\bar{2}, \bar{1}) \in S \). Then 
\[
\alpha_2^2 = s' \alpha_2, \qquad \alpha_4^2 = s' \alpha_4.
\]
Finally, for the element \( \alpha_5 \), take \( s'' = (\bar{2}, \bar{2}) \in S \), and we find that 
\[
\alpha_5^2 = s'' \alpha_5.
\]
Clearly, we cannot find a single \( s\in S \) such that each non-idempotent element of \( R \) is an \( S \)-idempotent with respect to \( s \).  
However, for every \( a \in R \), there exists some \( s \in S \) such that \( a \) is an \( S \)-idempotent element of \( R \).  
Thus \( R \) is not a uniformly \( S \)-Boolean ring, but it is an \( S \)-Boolean ring.
\end{exmp}

Here is an example of a uniformly \( S \)-von Neumann regular ring which is not uniformly \( S \)-Boolean:

\begin{exmp}[\cite{Zha4}, Example 3.18]\label{exmp-u-s-von-neum-not-u-s-bool}
Let \( T = \mathbb{Z}_2 \times \mathbb{Z}_2 \) and \( s = (1, 0) \in T \).  
Let \( R = T[x]/\langle sx, x^2 \rangle \), where \( x \) is an indeterminate, and let \( S = \{1, s\} \) be a multiplicative subset of \( R \).  
Then, by \cite[Example 3.18]{Zha4}, \( R \) is a uniformly \( S \)-von Neumann regular ring, but \( R \) is not von Neumann regular.  
Moreover, \( R \) is not uniformly \( S \)-Boolean; indeed, if \( R \) were uniformly \( S \)-Boolean, then it must be with respect to \( s \), since \( R \) is not Boolean.  
But
\[
s(1+\overline{x}) = s+s\overline{x} = (1+\overline{x})^2 = 1+\overline{x}^2 = 1+s\overline{x},
\]
which implies \( 1=s \), a contradiction.  
Thus \( R \) is not uniformly \( S \)-Boolean as desired. Furthermore, observe that \(R_S = R\) is not a Boolean ring. This shows that the condition "every element of \(R\) is \(S\)-idempotent" does not necessarily imply that the localization \(R_S\) is Boolean, in agreement with the remark preceding Definitions~\ref{section-3-defns}.
\end{exmp}

For the other counterexamples, we will use the following result.

\begin{prop}\label{prop-product-change-result-for-rings} 
Let \( R \) be the direct product of a family of commutative rings \( \{R_i\}_{i\in I} \) and let \( S \) be a multiplicative subset of \( R \).  
Then 
\begin{enumerate}
\item \( R \) is (uniformly) \( S \)-Boolean if and only if, for each \( i\in I \), \( R_i \) is (uniformly) \( S_i \)-Boolean. 
\item \( R \) is uniformly \( S \)-von Neumann regular if and only if, for each \( i\in I \), \( R_i \) is uniformly \( S_i \)-von Neumann regular. 
\item If \( I \) is finite, then \( R \) is uniformly \( S \)-\( \pi \)-regular if and only if, for each \( i\in I \), \( R_i \) is uniformly \( S_i \)-\( \pi \)-regular.
\end{enumerate}
\end{prop}

\begin{proof}
This follows immediately from Proposition~\ref{prop-product-change-uni-result}.
\end{proof}

We now give an example of a uniformly \( S \)-\( \pi \)-regular ring which is not \( \pi \)-regular.

\begin{exmp}\label{exmp-u-s-pi-not-pi}
Let \( R' \) be a \( \pi \)-regular ring and let \( R'' \) be a ring which is not \( \pi \)-regular.  
Set \( R=R'\times R'' \) and \( S=S'\times \{0\} \), where \( S' \) is any multiplicative subset of \( R' \).  
Then \( R \) is uniformly \( S \)-\( \pi \)-regular with respect to \( (1,0) \), but it is not \( \pi \)-regular.
\end{exmp}

Finally, we give an example of a uniformly \( S \)-\( \pi \)-regular ring which is not uniformly \( S \)-von Neumann regular.  Notice that a ring \( R \) is von Neumann regular (resp., \( \pi \)-regular) if and only if it is uniformly \( \u(R) \)-von Neumann regular (resp., uniformly \( \u(R) \)-\( \pi \)-regular), where \( \u(R) \) denotes the group of units of \( R \).

\begin{exmp}\label{exmp-u-s-pi-not-u-s-von-reg}
Let \( R' \) be a \( \pi \)-regular ring and let \( R'' \) be a \( \pi \)-regular ring which is not von Neumann regular.  
Set \( R=R'\times R'' \) and \( S=S'\times \u(R'') \), where \( S' \) is any multiplicative subset of \( R' \).  
Then \( R \) is uniformly \( S \)-\( \pi \)-regular with respect to \( (1,1) \), but it is not uniformly \( S \)-von Neumann regular.
\end{exmp}

Recall that a ring \(R\) is called \(S\)-Artinian if every descending sequence \(\{I_n\}_{n \in \mathbb{N}}\) of ideals of \(R\) is \(S\)-stationary. This means there exist a positive integer \(k\) and an element \(s \in S\) such that for all \(n \geq k\), \(sI_k \subseteq I_n\) \cite{Sev1}.  

A ring \(R\) is said to be uniformly \(S\)-Artinian if the element \(s\) in the definition above can be chosen independently of the descending sequence; that is, there exists \(s \in S\) such that every descending sequence \(\{I_n\}_{n \in \mathbb{N}}\) of ideals of \(R\) is \(S\)-stationary with respect to \(s\) \cite{Zha7}.  

\begin{prop}\label{r4}
The following statements hold true.
\begin{enumerate}
\item If \(R\) is \(S\)-Artinian, then every element of \(R\) is \(S\)-\(\pi\)-regular; equivalently, \(R_S\) is a \(\pi\)-regular ring.  
\item Any uniformly \(S\)-Artinian ring is uniformly \(S\)-\(\pi\)-regular.  
\end{enumerate}
\end{prop}

\begin{proof}
1. Assume that \(R\) is \(S\)-Artinian. Let \(r \in R\) and consider the descending chain of ideals
\[
Rr \supseteq Rr^2 \supseteq Rr^3 \supseteq \cdots .
\]
By the \(S\)-Artinian property, there exist a positive integer \(k\) and an element \(s \in S\) such that for all \(n \geq k\), $sRr^k \subseteq Rr^n$. In particular, \(sRr^k \subseteq Rr^{2k}\), which shows that \(r\) is \(S\)-\(\pi\)-regular. The equivalence with \(R_S\) being \(\pi\)-regular follows from the discussion preceding Definitions~\ref{section-3-defns}.  

2. This follows by a similar argument.
\end{proof}

Of course, a uniformly $S$-$\pi$-regular ring need not be uniformly $S$-Artinian; one can construct a counterexample using arguments similar to those in Examples~\ref{exmp-u-s-pi-not-pi} and \ref{exmp-u-s-pi-not-u-s-von-reg}.

Recall from \cite[Theorem 2.6]{And1} that a commutative ring $R$ is Boolean if and only if every element of $R$ is either idempotent or nilpotent; that is,  $R = \idem(R) \cup \nil(R)$. Replacing Boolean with $S$-Boolean, we have:

\begin{prop}\label{s-idem-s-nil}
Let $S$ be a multiplicative subset of $R$ containing no units except $1$. Then $ R = \Sidem(R) \,\cup\, \nil(R)$ if and only if $R$ is an $S$-Boolean ring.
\end{prop}

\begin{proof}
Assume that $R = \Sidem(R) \cup \nil(R)$. Then $U(R) \subseteq \Sidem(R)$. Let $a \in U(R)$. Then $a^2 = sa$ for some $s \in S$, which implies $a = s$, so $U(R) = \{1\}$. Consequently, $\nil(R) = \{0\}$ since $U(R) + \nil(R) = U(R)$. Therefore, $R = \Sidem(R)$, and hence $R$ is an $S$-Boolean ring. The converse is immediate.
\end{proof}

 It is not clear  whether the condition on \(S\) in the Proposition \ref{s-idem-s-nil} can be removed, nor whether \(\nil(R)\) can be replaced with \(\Snil(R)\); that is, if \(R = \Sidem(R) \cup \Snil(R)\), is \(R\) necessarily an \(S\)-Boolean ring? This question remains an open problem.
 
 Another result from \cite[Theorem 2.7]{And1} states that for a commutative ring $R$ with at least one nonzero zero-divisor, it suffices for $R$ to be Boolean that every zero-divisor is idempotent. Currently, we are unable to prove an $S$-version of this result; however, we can state the following partial result:

\begin{prop}
Let $s \in S$. Assume that $Z(R)$ contains a nonzero element $x \neq s$. If $\Zero(R) \subseteq s\text{-}\idem(R)$, then $s^2 y \in S$ for every $y \in R \setminus Z(R)$. Consequently, $R$ is uniformly $S$-von Neumann regular.
\end{prop}

\begin{proof}
Let \(x \in Z(R) \setminus \{0\}\). Since \(x \in s\text{-}\idem(R)\), we have \(x^2 = sx\).  
By Proposition~\ref{prop-on-interse-of-some-special-subset}, we also have $s(s-x) = (s-x)^2$.

Now let \(y \in R \setminus Z(R)\). Then both \(xy\) and \((s-x)y\) belong to \(Z(R)^* \subseteq s\text{-}\idem(R)\), so
\[
s xy = (xy)^2 = sxy^2 \quad \text{and} \quad s (s-x)y = ((s-x)y)^2 = s (s-x)y^2.
\]

Since \(y\) is regular, it follows that
\[
sx = sxy \quad \text{and} \quad s(s-x) = s (s-x)y.
\]

Adding these equalities, we obtain
\[
s^2 = sx + s(s-x) = sxy + s(s-x)y = (sx + s(s-x))y = s^2 y.
\]

Hence, for any \(y \in R \setminus Z(R)\), we have \(s^2y = s^2y^2\). Therefore, \(R\) is uniformly \(S\)-von Neumann regular with respect to \(s^2\).
\end{proof}




It is well known that, over a commutative Boolean ring, or more generally over a von Neumann regular ring, every primary ideal is maximal (and hence prime). In Theorem~\ref{pri-max-bool}, we extend this result to the $S$-setting.

Recall from \cite[Definition~2.1]{me22} that a proper ideal \(P\) of a ring \(R\), disjoint from \(S\), is said to be \emph{\(S\)-primary} if there exists \(s \in S\) such that for all \(a, b \in R\), whenever \(ab \in P\), either \(sa \in P\) or \(sb \in \rad(P)\), where \(\rad(P)\) denotes the radical of \(P\).

Also recall that any $S$-maximal ideal is $S$-prime \cite[Proposition 10]{Yil1}.

\begin{thm}\label{pri-max-bool}
Let $R$ be a (uniformly) $S$-Boolean ring and let $P$ be an $S$-primary ideal of $R$. Then $P$ is an $S$-maximal ideal of $R$, and hence $S$-prime.
\end{thm}
\begin{proof}
Suppose \(P\) is \(S\)-primary with respect to some \(t \in S\). Let \(Q\) be an ideal such that \(P \subseteq Q\) and \(Q \cap S = \emptyset\).  We claim that \(tQ \subseteq P\).  

Take \(x \in Q\). Since \(R\) is \(S\)-Boolean, there exists \(s \in S\) such that \(x^{2} = sx\).  
Hence, \((s-x)x = 0 \in P\). By the \(S\)-primaryness of \(P\), either \(tx \in P\) or \((t(s-x))^n \in P\) for some \(n \geq 1\).  

Assume that \((t(s-x))^n \in P\). Then, since \(P \subseteq Q\), we also have \((t(s-x))^n \in Q\).  
But note that $(t(s-x))^n = t^n s^n + r x$ for some \(r \in R\).   As \(x \in Q\), it follows that \(t^ns^n = (t(s-x))^n - rx \in Q\).  But \(t^ns^n \in S\), so \(t^ns^n \in Q \cap S\), contradicting the assumption that \(Q \cap S = \emptyset\).  Thus, we must have $tx \in P$.  

Therefore, \(tQ \subseteq P\), which shows that \(P\) is \(S\)-maximal.
\end{proof}

\noindent
By following similar arguments as in the proof of Theorem~\ref{pri-max-bool}, one can show the same result under the weaker assumption that $R$ is uniformly $S$-von Neumann regular (or even just that every element of $R$ is $S$-von Neumann regular).  

We conclude with the following result, which can be seen as an $S$-version of primary decomposition over Boolean rings.

\begin{thm}\label{r3} Assume that $R$ is uniformly $S$-Boolean with respect to some $s \in S$ with $s^2=s$.
If $I$ is an ideal of $R$ disjoint from $S$, then 
\[
s \cdot \bigcap_{\substack{M \supseteq I \\ M \in \Max_{S}(R)}} M \;\subseteq\; I \;\subseteq\; \bigcap_{\substack{M \supseteq I \\ M \in \Max_{S}(R)}} M,
\]
where $\Max_{S}(R)$ denotes the set of all $S$-maximal ideals of $R$.
\end{thm}

\begin{proof}
 Set
\[
J = \bigcap_{\substack{M \supseteq I \\ M \in \Max_{S}(R)}} M.
\]
Clearly, $I \subseteq J$. We claim that $sJ \subseteq I$. Suppose, for contradiction, that $sJ \nsubseteq I$. Then there exists $x \in J$ with $sx \notin I$. Consider the family of ideals
\[
\mathcal{F} = \{\, K \subseteq R \mid I \subseteq K, \; sx \notin K, \; K \cap S = \emptyset \,\}.
\]
Clearly, $\mathcal{F} \neq \emptyset$ since $I \in \mathcal{F}$. Moreover, every chain in $\mathcal{F}$ has an upper bound in $\mathcal{F}$ (the union). By Zorn's lemma, $\mathcal{F}$ has a maximal element, say $M$.

We now show that $M$ is an $S$-prime ideal. Let $a, b \in R$ with $ab \in M$. If neither $sa \in M$ nor $sb \in M$, then by maximality of $M$ in $\mathcal{F}$ we must have
\[
sx \in M + Rsa \quad \text{and} \quad sx \in M + Rsb,
\]
so
\[
sx = m + r_1 sa \quad \text{and} \quad sx = m' + r_2 sb
\]
for some $m, m' \in M$ and $r_1, r_2 \in R$. Hence $sx^2 \in M$. Since \(R\) is uniformly \(S\)-Boolean, we have either \(x^2 = sx\) or \(x^2 = x\).  In both cases, it follows that \(sx \in M\), contradicting the fact that  \(M\in \mathcal{F}\). Therefore, either $sa \in M$ or $sb \in M$, and thus $M$ is $S$-prime with respect to $s$.

By Theorem~\ref{pri-max-bool}, in a uniformly $S$-Boolean ring every $S$-prime ideal is $S$-maximal. Hence $M$ is $S$-maximal. But then $sx \in J \subseteq M$, which contradicts $sx \notin M$.  

Therefore, $sJ \subseteq I$, completing the proof.
\end{proof}

Driss Bennis:  Faculty of Sciences, Mohammed V University in Rabat, Rabat, Morocco.

\noindent e-mail address: driss.bennis@um5.ac.ma; driss$\_$bennis@hotmail.com

Ayoub Bouziri: Faculty of Sciences, Mohammed V University in Rabat, Rabat, Morocco.

\noindent e-mail address: ayoubbouziri66@gmail.com

Shiv Datt Kumar: Department of Mathematics, Motilal Nehru National Institute of Technology Allahabad, Prayagraj-211004, India

\noindent e-mail address: sdt@mnnit.ac.in

Tushar Singh(Corresponding Author): Department of Mathematics, Motilal Nehru National Institute of Technology Allahabad, Prayagraj-211004, India\\
\noindent e-mails address: sjstusharsingh0019@gmail.com, tushar.2021rma11@mnnit.ac.in, 
\end{document}